\documentclass{article}
\usepackage{amsfonts,amsmath,amssymb,amsthm,cite,enumerate,graphicx,latexsym,mathrsfs}
\usepackage{authblk}
\usepackage{hyperref}
\usepackage[T1]{fontenc}
\usepackage[utf8]{inputenc}
\usepackage{float}
\usepackage{geometry}
\usepackage{setspace}

\allowdisplaybreaks[2]

\date{}
\makeatletter

\newcommand{\Rmnum}[1]{\expandafter\@slowromancap\romannumeral #1@}
\makeatother

\numberwithin{equation}{section}

\newtheorem{lemma}{Lemma}[section]

\newtheorem{remark}{Remark}[section]

\newtheorem{theorem}{Theorem}[section]

\newcommand\theref[1]{Theorem~\ref{#1}}

\newcommand\secref[1]{Section~\ref{#1}}

\title{Existence of detonation wave solutions to the piston problem for the Zeldovich-von Neumann-D{\"o}ring combustion model}
\author{Xiaomin Zhang\thanks{{e}-mail: zxm15924687@163.com} \quad Huimin Yu\thanks{Corresponding author {{e}-mail: hmyu@sdnu.edu.cn}}
 \\ \small\textit{ Department of mathematics, Shandong Normal University, Jinan 250014 China}}
\begin{document}
\begin{sloppypar}
\onehalfspacing
\date{}
\maketitle

\begin{center}
\begin{minipage}{130mm}{\small
\textbf{Abstract}:
In this paper, we study detonation wave solutions to one-dimensional piston problem for the Zeldovich-von Neumann-D{\"o}ring (ZND) combustion model with a one-step exothermic chemical reaction. As a special type of shock wave, the position of the detonation wave is unknown, which make our model to be a free boundary problem.~The global existence of detonation wave solutions to this free boundary problem is proved. Compared with previous result~\cite{Lai}, we do not impose any dissipation conditions on the equations and the boundaries.
\\
\textbf{Keywords}: Detonation wave, Piston problem, Zeldovich-von Neumann-D{\"o}ring combustion model, Global existence
\\
\textbf{Mathematics Subject Classification 2020}: 35L65; 35L67; 76N10.}
\end{minipage}
\end{center}

\section{Introduction}
\indent\indent In a piston-driven combustion system for flammable gases, the high-speed motion of the piston serves as the core driving source for shock wave generation. In this paper, assume that the piston is located at the origin initially and then suddenly moves with a speed depending only on time $t$. When the piston moves into the unburnt gas, a shock front emerges and propagates away from the piston. The gas ahead of this shock front remains in a static and unburnt state. As the gas passes through the shock front, there is a notable increase in temperature. When the temperature rise across the shock front surpasses the ignition temperature of the gas, a combustion reaction is triggered, occurring in the region behind the shock front. The pre-compression shock wave in the combustion process is commonly referred to as a detonation wave, as documented in the work of Courant and Friedrichs (see ~\cite{Courant}). In fact, combustion waves can be partitioned into two primary types. One is the detonation wave, which features supersonic propagation and a compressive nature. The other is the deflagration wave, characterized by subsonic motion and an expansive behavior. Since chemical reactions are governed by molecular collisions, the length scale relevant to the chemical reaction is generally substantially larger than that of the shock wave associated with detonation. As a result, for high Mach number combustion, the Zeldovich-von Neumann-D{\"o}ring (ZND) model stands out as the most suitable inviscid model. We study the denotation solution to one-dimensional piston problem for the ZND combustion model in this paper. As an important model for describing the propagation of combustion waves in a one-step exothermic chemical reaction, the ZND model has a finite reaction rate and no viscous and heat conduction effects. It is expressed as the following exothermically reacting compressible Euler equations in the Lagrangian coordinates:
\begin{equation}\label{a1}
\left\{\begin{aligned}
&\nu_{t}-u_{x}=0,\\
&u_{t}+p_{x}=0,\\
&E_{t}+(pu)_{x}=0,\\
&Z_{t}=-\kappa\psi(T)Z,
\end{aligned}\right.
\end{equation}
where $\nu, u, p,Z$ represent the specific volume, velocity, pressure and fraction of unburned gas in the mixture respectively. $E=\frac{1}{2}u^{2}+e+Z\hbar$ is the specific total energy with the specific internal energy $e$ and the binding energy $\hbar>0$ per unit mass of unburnt gas. $\psi(T)$ is the ignition function with the temperature $T$ satisfying the Arrhenius-type:
\begin{align*}
\psi(T)=
\left\{\begin{aligned}
&T^{\ell}e^{-\frac{\mathcal{A}}{T-T_{i}}}, \quad T>T_{i},\\
&0,\quad\quad\quad\quad~~ T\leq T_{i},
\end{aligned}\right.
\end{align*}
where the constant $T_{i}>0$ is the ignition temperature and $\ell, \mathcal{A}$ are two constants.
The constant $\kappa>0$ represents the reaction rate.

In this paper, we consider the internal energy $e$ satisfying the following equation of state:
\begin{align}\label{a3}
e=e(\nu,s)=
\left\{\begin{aligned}
&\frac{e^{s}\nu^{1-\gamma}}{\gamma-1},\quad\quad\quad\quad\quad\qquad\quad\quad\quad\quad~~~~ \nu\geq1,\\
&\frac{\gamma}{2}e^{s}\nu^{2}-(\gamma+1)e^{s}\nu+\frac{\gamma^{2}+\gamma}{2(\gamma-1)}e^{s},\quad \nu<1,
\end{aligned}\right.
\end{align}
where $s$ represents the specific entropy and $\gamma>1$ is a constant. Since the internal state variables $e,T,s,p,\nu$ must satisfy the thermodynamical relation
\begin{align}
de=Tds-pd\nu,\label{tt1}
\end{align}
we obtain
\begin{align}\label{a4}
p=p(\nu,s)=
\left\{\begin{aligned}
&e^{s}\nu^{-\gamma},\quad\quad\quad\quad\quad~~ \nu\geq1,\\
&(\gamma+1)e^{s}-\gamma e^{s}\nu,\quad \nu<1,
\end{aligned}\right.
\end{align}
and
\begin{align}
T=e.\label{a5}
\end{align}

Over the past few years, significant progress has been made in the research on shock wave solutions to piston problems. In 2003, Chen~\cite{ChenS} carried out a study on the global existence and stability of shock front solutions for the multi-dimensional piston problem for the unsteady potential flow equation. Subsequently, Chen $\emph{et al.}$~\cite{ChenGGG} investigated the local shock solution for a multi-dimensional piston problem for the compressible Euler equations. In 2008, Xu and Dou~\cite{Xu} probed into the global existence of shock front solutions for the one-dimensional piston problem of the relativistic Euler equations, assuming that the piston velocity is a perturbation of a constant. Then, Chen $\emph{et al.}$~\cite{ChenSS} studied the global existence of shock front solution to the 2-dimensional axially symmetric piston problem for compressible Euler equations. In 2013, Ding and Li~\cite{Ding} investigated the global existence and non-relativistic limits of weak solutions for the one-dimensional piston problem of the isentropic relativistic Euler equations when the total variations of the initial data and the piston velocity were suitably small. Then, in 2018, Ding~\cite{DingM} explored the global existence of shock front solutions for the one-dimensional piston problem of the compressible Euler equations under the same conditions. In 2022, Lai~\cite{LaiG} conducted research on the global existence of shock front solutions for a spherical piston problem of the relativistic Euler equations.
%In 2025, Zhang and Yu~\cite{ZhangX} studied the existence and stability of the contact discontinuity to a piston problem for one-dimensional compressible Euler equations.

At the same time, extensive research also has been conducted on the exothermically reacting Euler equations. For weak entropy solutions in $BV$ or $L^{\infty}$ space, Chen and Wagner~\cite{Chen} delved into the global existence of entropy solutions for the one-dimensional compressible Euler equations involving a one-step exothermic chemical reaction. Then, Chen $\emph{et al.}$~\cite{ChenG} explored the global existence of entropy solutions of the two-dimensional steady exothermically reacting Euler equations. Subsequently, Chen $\emph{et al.}$~\cite{ChenGG} investigated the existence of global entropy solutions with a strong rarefaction wave to the two-dimensional steady supersonic reacting Euler flow past Lipschitz bending walls, under the condition that the total variation of both the initial data and the slope of the boundary is sufficiently small. Then, Xiang $\emph{et al.}$~\cite{Xiang} studied the existence of the global entropy solution containing a strong contact discontinuity under the same model and assumptions. Hu~\cite{Hu} focused on the stability and uniqueness of the global entropy solution for the Cauchy problem of the exothermically reacting compressible Euler equations. Regarding the piston problem, Kuang and Zhao~\cite{Kuang} explored the global existence and stability of entropy solutions including a strong shock front wave to one-dimensional piston problem of the exothermically reacting Euler equations. Then, Hu and Kuang~\cite{HuK} extended the result in~\cite{Kuang} to ZND combustion model. Besides, Lai~\cite{Lai} examined the global existence of the detonation wave solutions to a one-dimensional piston problem for the ZND combustion model with a one-step exothermic chemical reaction in $C^{1}$ space, where some dissipative conditions were imposed on the equations and the boundary. Fang $\emph{et al.}$~\cite{Fang} probed into the existence of transonic shocks for steady exothermically reacting Euler flows with a small exothermic reaction rate in an almost flat nozzle in a Sobolev space later. Most recently, Zhang~\cite{Zhang} investigated the existence and uniqueness of smooth subsonic flows for the three-dimensional steady ZND combustion equations in a cylindrical nozzle. Readers seeking more in-depth knowledge about the reacting dynamic theory are referred to the works cited in~\cite{Majda,Williams}.

In this paper, we reconsider the global existence of detonation wave solutions, which are piecewise smooth, to one-dimensional piston problem for the ZND combustion model.
Assume that the unburnt gas is at rest initially with a uniform state
\begin{align}
(\nu,u,p,Z)(0,x)=(\nu_{0},0,p_{0},1),\quad x>0,\label{a7}
\end{align}
where $\nu_{0}>\frac{\gamma+1}{\gamma}, p_{0}>0$ are some constants.
The boundary condition on $x=0$ is
\begin{align}
x=0:~~ u(t,0)=B'(t),\quad t\geq0,\label{a8}
\end{align}
where $B'(t)>0$ is the given function denoting the piston velocity.

The same problem had been studied in~\cite{Lai}, where the condition (A1): $$\nu_{p}+4\mathcal{A}_{p}\mu_{p}<1$$
was imposed to ensure the governing equations and boundary have some dissipative structure. Here, $\nu_p$ denotes the derivative of the Riemann invariant impinging on the shock along characteristics with respect to another Riemann invariant, evaluated at the piecewise constant shock solution state, $\mu_p$ represents the entropy derivative with respect to the same Riemann invariant at this solution state, while $\mathcal{A}_p$ is the coefficient of linear source terms in the governing equations at the aforementioned solution state. In the proof, the author diagonalized the equations via Riemann invariants and derived the uniform estimates of the solution through the diagonalized system using triangle inequalities. Since the source terms in the reduced system included linear terms, which made the author have to apply additional conditions to guarantee interior dissipation of the governing equations. In contrast, we instead consider the perturbed equations and left-multiply them by eigenvectors to derive another diagonalized system. Fortunately, the right-hand side of the resulting linearized diagonalized system is composed solely of nonlinear terms, which inherently offer sufficient dissipation for small perturbations. Furthermore, we exploit the Rankine-Hugoniot conditions to extract inherent dissipation at the shock wave, rendering the additional constraint (A1) unnecessary. Moreover, our method permits the piston velocity to be a time-dependent function $B'(t)$, while it must be a constant in~\cite{Lai}.

The subsequent arrangement of this paper is as follows: In~\secref{s2}, we elaborate in detail on the existence of piecewise constant shock wave solutions to the piston problem~\eqref{a1},\eqref{a7}-\eqref{a8} and present the main result of this paper. Subsequently, we carried out certain transformations on both the equations and the boundaries.~These transformations converted the equations into a diagonal form with dissipative structure on both boundaries.~In~\secref{s3}, we utilized local existence and a priori estimates to prove~\theref{t2}, which in turn led to the proof of~\theref{t1}.

\section{Problem presentation and main result}\label{s2}
\indent\indent In this section, we will give our main result and derive the equivalent system.

\subsection{Problem presentation}\label{ss1}

Let the shock curve be $x=\chi(t)$. Then, the Rankine-Hugoniot conditions across the shock are
\begin{align}\label{a9}
\left\{\begin{aligned}
&[u]=-\chi'(t)[\nu],\\
&[p]=\chi'(t)[u],\\
&[pu]=\chi'(t)[E],
\end{aligned}\right.
\end{align}
where $[h]$ stands for the jump of function $h$ across the shock wave. The Lax geometric entropy condition across the shock wave is (c.f.~\cite{Lai})
\begin{align}
\sqrt{-p_{\nu}(\nu_{0},s_{0})}<\chi'(t)<\sqrt{-p_{\nu}(\nu,s)},\label{a10}
\end{align}
where $s_{0}$ is the initial entropy, which can be determined by~\eqref{a4} and~\eqref{a7}.
%From the Rankine-Hugoniot conditions~\eqref{a9}, we have
%\begin{align}
%&u=-\omega(\nu-\nu_{0}),\label{bb1}\\
%&p-p_{0}=\omega u,\label{bb2}\\
%&pu=\omega\big(E(\nu,u,p)-E(\nu_{0},0,p_{0})\big),\label{bb3}
%\end{align}
%for the shock velocity $\omega>0$ and any given initial static condition $(\nu_{0},0,p_{0},1)$.
Obviously, we have $\nu<\nu_{0}$.
%Moreover, The relationships presented in~\eqref{a9} define a curve that is parameterized by a single variable in the $(\nu,u,p,Z)$-space. This curve is referred to as the shock curve and is symbolized by $\Xi$.

From (3.1) in~\cite{Lai}, we obtain
\begin{align}
\frac{du}{d\nu}|_{\Xi}<0,\quad \frac{de}{d\nu}|_{\Xi}=\frac{dT}{d\nu}|_{\Xi}<0.\label{b10}
\end{align}

Suppose that along the shock curve $x=\chi(t)$, $u=u_{o}$, and $p=p_{o}$ as $\nu=0$. Then, for any velocity $u_{\iota}\in(0, u_{o})$, there exist $\nu_{\iota}$ and $p_{\iota}$ such that the state $(\nu_{\iota}, u_{\iota}, p_{\iota})$ lies on the shock curve $x=\chi(t)$. As a consequence, the states $(\nu_{0},0,p_{0})$ and $(\nu_{\iota},u_{\iota},p_{\iota})$ can be connected by a forward shock wave. Next, assume that when $\nu = 1$ on the shock curve $x=\chi(t)$, we have $u=u_{1}$ and $p=p_{1}$. It follows from~\eqref{b10} that if $u_{\iota}>u_{1}$, we have $\nu_{\iota}<1$. In this paper, we assume $u_{\iota}$ satisfying
\begin{align}
u_{1}<u_{\iota}<u_{o}.\label{b11}
\end{align}
From~\eqref{b10}, we can infer that the temperature is increasing across the shock wave. If $T(\nu_{\iota},p_{\iota})\leq T_{i}$, then when the piston moves at a constant velocity $b_{0}=u_{\iota}$, the piston problem~\eqref{a1},\eqref{a7}-\eqref{a8} admits a constant shock wave solution $(\nu_{b},u_{b},p_{b},1)$ with a shock wave $x=\chi(t)\equiv\chi_{0}t$, which has the following form:
\begin{align}\label{b12}
(\nu_{b},u_{b},p_{b},1)=
\left\{\begin{aligned}
&(\nu_{\iota},u_{\iota},p_{\iota},1),\quad t>0,0<x<\chi_{0}t,\\
&(\nu_{0},0,p_{0},1),\quad t>0,x>\chi_{0}t,
\end{aligned}\right.
\end{align}
where
\begin{align}
\chi_{0}=\sqrt{-\frac{p_{\iota}-p_{0}}{\nu_{\iota}-\nu_{0}}}>\max\{b_{0}, c_{0}\}\label{BB12}
\end{align}
with $c_{0}=\sqrt{-p_{\nu}(\nu_{0},s_{0})}$.
We adopt this constant shock wave solution $(\nu_{b}, u_{b}, p_{b}, 1)$ as the background solution. Subsequently, we aim to investigate the shock wave solution of the piston problem~\eqref{a1},\eqref{a7}-\eqref{a8} when the piston velocity is a time-dependent function $B'(t)$.

On the other hand, if $T(\nu_{\iota}, p_{\iota})>T_{i}$, a combustion reaction will be triggered. In this case, we will search for a detonation wave solution of the piston problem~\eqref{a1},\eqref{a7}-\eqref{a8}.

Formally, our main result is encapsulated in the following theorem:
\begin{theorem}\label{t1}
There exist two constants $\hbar^{*}>0$ and $\epsilon_{0}>0$, such that for any given $\epsilon\in(0,\epsilon_{0})$ and any given $\hbar\in(0,\hbar^{*})$, if the piston velocity $B'(t)$ satisfies $B'(0)=b_{0}>0$ and
\begin{align}
\|B'(t)-b_{0}\|_{C^{1}}\leq\epsilon,\label{a11}
\end{align}
then  for any reaction rate $\kappa>0$, the piston problem~\eqref{a1},\eqref{a7}-\eqref{a8} admits a unique global in time shock or detonation wave solution $(\hat{\nu},\hat{u},\hat{p},\hat{Z})(t,x)$, which has the following form
\begin{align}\label{a12}
(\hat{\nu},\hat{u},\hat{p},\hat{Z})(t,x)=
\left\{\begin{aligned}
&(\nu,u,p,Z)(t,x),\quad (t,x)\in \hat{D}_{l}=\{(t,x)|t\in\mathbb{R}_{+}, 0\leq x<\chi(t)\},\\
&(\nu_{0},0,p_{0},1),\quad\quad~ (t,x)\in \hat{D}_{r}=\{(t,x)|t\in\mathbb{R}_{+}, x>\chi(t)\}.
\end{aligned}\right.
\end{align}
Moreover, this shock or detonation wave solution includes a shock wave $x=\chi(t)$ and satisfies the following estimate:
\begin{align}
\|(\nu,u,p)(t,x)-(\nu_{\iota},u_{\iota},p_{\iota})\|_{C^{1}(\hat{D}_{l})}&\leq C_{E}\epsilon,\label{a13}\\
\|Z(t,x)-1\|_{C^{1}(\hat{D}_{l})}&\leq C_{E},\label{a14}\\
\|\chi'(t)-\chi_{0}\|_{C^{1}(\mathbb{R}_{+})}&\leq C_{E}\epsilon,\label{a15}
\end{align}
where $C_{E}>0$ is a constant independent of $\epsilon$.
\end{theorem}

\subsection{The derivation of equivalent systems}\label{ss2}
\indent\indent Since the gas is at rest and unburnt on the right side of the shock curve $x=\chi(t)$, we only need to discuss the state of gas on the left side of the shock curve $x=\chi(t)$. In the following derivations, we assume $\nu<1$.

By $E=\frac{1}{2}u^{2}+e+Z\hbar$ and~\eqref{a3}-\eqref{a5}, equations~\eqref{a1} can be rewritten as
\begin{align}\label{b13}
\left\{\begin{aligned}
&\nu_{t}-u_{x}=0,\\
&s_{t}=\frac{1}{T}\kappa\psi(T)Z\hbar,\\
&u_{t}-\gamma e^{s}\nu_{x}+\big((\gamma+1)e^{s}-\gamma e^{s}\nu\big)s_{x}=0.
\end{aligned}\right.
\end{align}
In accordance with~\cite{Lai}, for simplicity, we take $\psi(T)$ satisfying
\begin{align}\label{a2}
\psi(T)=
\left\{\begin{aligned}
&1, \quad T>T_{i},\\
&0,\quad T\leq T_{i}.
\end{aligned}\right.
\end{align}
\begin{remark}
It can be seen that when $T\leq T_{i}$, equations~\eqref{b13} become the homogeneous equations and the smallness of $\hbar$ is not required. Moreover, the research methods for the shock wave solutions in the two cases of $T\leq T_{i}$ and $T>T_{i}$ are similar. Therefore, we will only discuss the case of $T>T_{i}$ below.
\end{remark}

Define the perturbation
$$ \bar{\mathbf{\Phi}}=(\bar{\nu},\bar{s},\bar{u})^{\top}=(\nu-\nu_{\iota},s-s_{\iota},u-u_{\iota})^{\top}$$
with $s_{\iota}=\ln\big(\frac{p_{\iota}}{\gamma+1-\gamma\nu_{\iota}}\big)$,
we obtain
\begin{align}
\bar{\mathbf{\Phi}}_{t}+\bar{\Lambda}(\bar{\mathbf{\Phi}})\bar{\mathbf{\Phi}}_{x}=\bar{F}(\bar{\mathbf{\Phi}}),\label{b16}
\end{align}
where
\begin{align*}
\bar{\Lambda}(\bar{\mathbf{\Phi}})=
\begin{pmatrix}
&0 &0 &-1\\
&0 &0 &0\\
&-\gamma e^{s} &(\gamma+1)e^{s}-\gamma e^{s}\nu &1
\end{pmatrix},
\end{align*}
$$\bar{F}(\bar{\mathbf{\Phi}})=\Big(0,\frac{1}{T}\kappa Z\hbar,0\Big)^{\top}.$$
Noting that there exists a matrix
\begin{align*}
L=
\begin{pmatrix}
&-\sqrt{\gamma e^{s_{\iota}}} &\sqrt{\gamma e^{s_{\iota}}}(\frac{\gamma+1}{\gamma}-\nu_{\iota}) &-1\\
&0 &1 &0\\
&-\sqrt{\gamma e^{s_{\iota}}} &\sqrt{\gamma e^{s_{\iota}}}(\frac{\gamma+1}{\gamma}-\nu_{\iota}) &1
\end{pmatrix},
\end{align*}
such that $L\bar{\Lambda}(\mathbf{0})L^{-1}$ is a diagonal matrix. Then, let
\begin{align}
\mathbf{\Phi}=(\Phi_{1},\Phi_{2},\Phi_{3})^{\top}=L\bar{\mathbf{\Phi}},\label{BBB1}
\end{align}
system~\eqref{b16} becomes
\begin{align}
\mathbf{\Phi}_{t}+\Lambda(\mathbf{\Phi})\mathbf{\Phi}_{x}=F(\mathbf{\Phi}),\label{b17}
\end{align}
where
$$\Lambda(\mathbf{\Phi})=L\bar{\Lambda}(\bar{\mathbf{\Phi}})L^{-1},~~ F(\mathbf{\Phi})=L\bar{F}(\bar{\mathbf{\Phi}}).$$
Through simple mathematical calculations, we obtain that the eigenvalues of the matrix $\Lambda(\mathbf{\Phi})$ are
\begin{align}
\lambda_{1}=&\lambda_{1}(\Phi_{2})=-\sqrt{\gamma e^{\Phi_{2}+s_{\iota}}},\label{b18}\\
&\lambda_{2}=0,\label{b19}\\
\lambda_{3}=&\lambda_{3}(\Phi_{2})=\sqrt{\gamma e^{\Phi_{2}+s_{\iota}}}.\label{b20}
\end{align}
The left eigenvectors are
$$l_{1}=l_{1}(\mathbf{\Phi})=(1,b_{2},b_{3}),~~l_{2}=(0,1,0),~~l_{3}=l_{3}(\mathbf{\Phi})=(b_{3},b_{2},1),$$
where
\begin{align}
&b_{2}=b_{2}(\mathbf{\Phi})=\frac{\sqrt{\gamma e^{\Phi_{2}+s_{\iota}}}(\Phi_{1}+\Phi_{3})-2\sqrt{e^{s_{\iota}}}\sqrt{e^{\Phi_{2}+s_{\iota}}}\big((\gamma+1)-\gamma\nu_{\iota}\big)\Phi_{2}}{\sqrt{\gamma e^{\Phi_{2}+s_{\iota}}}+\sqrt{\gamma e^{s_{\iota}}}},\label{b24}\\
&b_{3}=b_{3}(\Phi_{2})=\frac{\sqrt{\gamma e^{\Phi_{2}+s_{\iota}}}-\sqrt{\gamma e^{s_{\iota}}}}{\sqrt{\gamma e^{\Phi_{2}+s_{\iota}}}+\sqrt{\gamma e^{s_{\iota}}}},\label{b25}
\end{align}
with
\begin{align}
b_{2}(\mathbf{0})=0,\quad b_{3}(0)=0.\label{b26}
\end{align}
By left-multiplying both sides of the system~\eqref{b17} by $l=(l_{1},l_{2},l_{3})^{\top}$ simultaneously, we get
\begin{align}
\frac{\partial\Phi_{1}}{\partial t}+\lambda_{1}(\Phi_{2})\frac{\partial\Phi_{1}}{\partial x}=&-b_{2}(\mathbf{\Phi})\big(\frac{\partial\Phi_{2}}{\partial t}+\lambda_{1}(\Phi_{2})\frac{\partial\Phi_{2}}{\partial x}\big)-b_{3}(\Phi_{2})\big(\frac{\partial\Phi_{3}}{\partial t}+\lambda_{1}(\Phi_{2})\frac{\partial\Phi_{3}}{\partial x}\big)\notag\\
&+\Big(\big(1+b_{3}(\Phi_{2})\big)\sqrt{\gamma e^{s_{\iota}}}\big(\frac{\gamma+1}{\gamma}-\nu_{\iota}\big)+b_{2}(\mathbf{\Phi})\Big)\frac{1}{T}\kappa Z\hbar,\label{b27}\\
\frac{\partial\Phi_{2}}{\partial t}=&\frac{1}{T}\kappa Z\hbar,\label{b28}\\
\frac{\partial\Phi_{3}}{\partial t}+\lambda_{3}(\Phi_{2})\frac{\partial\Phi_{3}}{\partial x}=&-b_{3}(\Phi_{2})\big(\frac{\partial\Phi_{1}}{\partial t}+\lambda_{3}(\Phi_{2})\frac{\partial\Phi_{1}}{\partial x}\big)-b_{2}(\mathbf{\Phi})\big(\frac{\partial\Phi_{2}}{\partial t}+\lambda_{3}(\Phi_{2})\frac{\partial\Phi_{2}}{\partial x}\big)\notag\\
&+\Big(\big(1+b_{3}(\Phi_{2})\big)\sqrt{\gamma e^{s_{\iota}}}\big(\frac{\gamma+1}{\gamma}-\nu_{\iota}\big)+b_{2}(\mathbf{\Phi})\Big)\frac{1}{T}\kappa Z\hbar.\label{b29}
\end{align}
The boundary condition~\eqref{a8} becomes
\begin{align}
x=0:~~\Phi_{3}(t,0)=\Phi_{1}(t,0)+2(B'(t)-b_{0}).\label{b31}
\end{align}
By using the Rankine-Hugoniot conditions~\eqref{a9}, we obtain the following boundary conditions on the shock wave $x=\chi(t)$
\begin{align}
x=\chi(t):~~&\Phi_{1}(t,\chi(t))=G_{1}\big(\Phi_{3}(t,\chi(t))\big),\label{b32}\\
&\Phi_{2}(t,\chi(t))=G_{2}\big(\Phi_{3}(t,\chi(t))\big),\label{b33}\\
&Z(t,\chi(t))=1.\label{b34}
\end{align}
The detailed proof of~\eqref{b32}-\eqref{b33} can be found in the following lemma.
\begin{lemma}
When $\gamma>1$ and $\nu_{0}>\frac{\gamma+1}{\gamma}$, there exist two functions $G_{1}(\Phi_{3})$ and $G_{2}(\Phi_{3})$, such that on the shock curve $x=\chi(t)$, we have
\begin{align}
\Phi_{1}(t,\chi(t))=G_{1}\big(\Phi_{3}(t,\chi(t))\big),\quad\Phi_{2}(t,\chi(t))=G_{2}\big(\Phi_{3}(t,\chi(t))\big),\label{b35}
\end{align}
which satisfy
\begin{align}
0<\Big|\frac{\partial G_{1}}{\partial\Phi_{3}}\big|_{\mathbf{\Phi}=0}\Big|<1.\label{bb35}
\end{align}
\end{lemma}
\begin{proof}
By the Rankine-Hugoniot conditions~\eqref{a9}, we get
\begin{align}
&J_{1}(\nu,s,u)=u-\sqrt{-\big((\gamma+1)e^{s}-\gamma e^{s}\nu-e^{s_{0}}\nu_{0}^{-\gamma}\big)(\nu-\nu_{0})}=0,\label{b36}\\
&J_{2}(\nu,s,u)=\frac{1}{2}e^{s}\Big((\gamma+1)(\nu+\nu_{0})-\frac{\gamma^{2}+\gamma}{\gamma-1}-\gamma\nu\nu_{0}\Big)+e^{s_{0}}
\nu_{0}^{-\gamma}\Big(\frac{\gamma+1}{2(\gamma-1)}\nu_{0}-\frac{1}{2}\nu\Big)=0.\label{b37}
\end{align}
Then
\begin{align}\label{b38}
\frac{\partial(J_{1},J_{2})}{\partial(\nu,s,u)}\Big|_{\bar{\mathbf{\Phi}}=0}=
\begin{pmatrix}
&\frac{\gamma e^{s_{\iota}}}{2\chi_{0}}+\frac{\chi_{0}}{2} &-\frac{(\gamma+1-\gamma\nu_{\iota})e^{s_{\iota}}}{2\chi_{0}} &1\\
&\frac{1}{2}(\gamma+1-\gamma\nu_{0})e^{s_{\iota}}-\frac{1}{2}e^{s_{0}}\nu_{0}^{-\gamma} &-e^{s_{0}}\nu_{0}^{-\gamma}\Big(\frac{\gamma+1}{2(\gamma-1)}\nu_{0}-\frac{1}{2}\nu_{\iota}\Big) &0
\end{pmatrix},
\end{align}
where
$$\chi_{0}=\sqrt{-\frac{(\gamma+1)e^{s_{\iota}}-\gamma e^{s_{\iota}}\nu_{\iota}-e^{s_{0}}\nu_{0}^{-\gamma}}{\nu_{\iota}-\nu_{0}}}.$$
On the other hand, it follows from~\eqref{BBB1}
\begin{align}
&\nu(\mathbf{\Phi})=-\frac{1}{2\sqrt{\gamma e^{s_{\iota}}}}(\Phi_{1}+\Phi_{3})+\Big(\frac{\gamma+1}{\gamma}-\nu_{\iota}\Big)\Phi_{2}+\nu_{\iota},\label{b21}\\
&s(\mathbf{\Phi})=\Phi_{2}+s_{\iota},\label{b22}\\
&u(\mathbf{\Phi})=\frac{1}{2}(\Phi_{3}-\Phi_{1})+u_{\iota}.\label{b23}
\end{align}
which means
\begin{align}\label{b42}
\frac{\partial(\nu,s,u)}{\partial(\Phi_{1},\Phi_{2})}\Big|_{\mathbf{\Phi}=0}=
\begin{pmatrix}
&-\frac{1}{2\sqrt{\gamma e^{s_{\iota}}}} &\frac{\gamma+1}{\gamma}-\nu_{\iota}\\\
&0 &1\\
&-\frac{1}{2} &0
\end{pmatrix}.
\end{align}
Denote
\begin{align*}
&K_{1}(\Phi_{1},\Phi_{2},\Phi_{3})=J_{1}(\nu(\mathbf{\Phi}),s(\mathbf{\Phi}),u(\mathbf{\Phi})),\\
&K_{2}(\Phi_{1},\Phi_{2},\Phi_{3})=J_{2}(\nu(\mathbf{\Phi}),s(\mathbf{\Phi}),u(\mathbf{\Phi})),
\end{align*}
Here, $\nu,s,u$ and $\mathbf{\Phi}$ are all at $(t,\chi(t))$. For the sake of simplicity, we omit writing it.

Then, combining with~\eqref{b38} and~\eqref{b42}, we have
\begin{align*}
&\frac{\partial(K_{1},K_{2})}{\partial(\Phi_{1},\Phi_{2})}\Big|_{\mathbf{\Phi}=0}=\frac{\partial(J_{1},J_{2})}
{\partial(\nu,s,u)}\Big|_{\bar{\mathbf{\Phi}}=0}\cdot\frac{\partial(\nu,s,u)}{\partial(\Phi_{1},\Phi_{2})}\Big|_{\mathbf{\Phi}=0}
=\big(k_{ij}\big)_{2\times2},
\end{align*}
where
\begin{align*}
&k_{11}=-\frac{\sqrt{\gamma e^{s_{\iota}}}}{4\chi_{0}}-\frac{\chi_{0}}{4\sqrt{\gamma e^{s_{\iota}}}}-\frac{1}{2},\\
&k_{12}=\Big(\frac{\gamma+1}{2\gamma}-\frac{\nu_{\iota}}{2}\Big)\chi_{0},\\
&k_{21}=-\frac{1}{4\sqrt{\gamma e^{s_{\iota}}}}\big((\gamma+1-\gamma\nu_{0})e^{s_{\iota}}-e^{s_{0}}\nu_{0}^{-\gamma}\big),\\
&k_{22}=\Big(\frac{\gamma+1}{2\gamma}-\frac{\nu_{\iota}}{2}\Big)\big((\gamma+1-\gamma\nu_{0})e^{s_{\iota}}-e^{s_{0}}\nu_{0}^{-\gamma}\big)
-e^{s_{0}}\nu_{0}^{-\gamma}\frac{(\gamma+1)\nu_{0}-(\gamma-1)\nu_{\iota}}{2(\gamma-1)}.
\end{align*}
Thus, by $\gamma>1$, $\nu_{0}>\frac{\gamma+1}{\gamma}$ and $\nu_{\iota}<1$, we have
\begin{align}
&\det\Big(\frac{\partial(K_{1},K_{2})}{\partial(\Phi_{1},\Phi_{2})}\Big|_{\mathbf{\Phi}=0}\Big)=k_{11}k_{22}-k_{12}k_{21}\notag\\
=&-\frac{\gamma+1-\gamma\nu_{\iota}}{2\gamma}\big(\frac{\sqrt{\gamma e^{s_{\iota}}}}{4\chi_{0}}+\frac{1}{2}\big)\big((\gamma+1-\gamma\nu_{0})e^{s_{\iota}}-e^{s_{0}}
\nu_{0}^{-\gamma}\big)\notag\\\
&+e^{s_{0}}\nu_{0}^{-\gamma}\frac{(\gamma+1)\nu_{0}-(\gamma-1)\nu_{\iota}}{2(\gamma-1)}\frac{(\sqrt{\gamma e^{s_{\iota}}}+\chi_{0})^{2}}{4\chi_{0}\sqrt{\gamma e^{s_{\iota}}}}>0,\label{b43}
\end{align}
which indicates that we can employ the existence theorem of the implicit function to obtain~\eqref{b35}.

Similarly, we have
\begin{align*}
\frac{\partial(K_{1},K_{2})}{\partial(\Phi_{3},\Phi_{2})}\Big|_{\mathbf{\Phi}=0}=\frac{\partial(J_{1},J_{2})}
{\partial(\nu,s,u)}\Big|_{\bar{\mathbf{\Phi}}=0}\cdot\frac{\partial(\nu,s,u)}{\partial(\Phi_{3},\Phi_{2})}\Big|_{\mathbf{\Phi}=0}
=\big(\tilde{k}_{ij}\big)_{2\times2}
\end{align*}
with
$$
\tilde{k}_{11}=-\frac{\sqrt{\gamma e^{s_{\iota}}}}{4\chi_{0}}-\frac{\chi_{0}}{4\sqrt{\gamma e^{s_{\iota}}}}+\frac{1}{2},~~\tilde{k}_{12}=k_{12},~~ \tilde{k}_{21}=k_{21},~~ \tilde{k}_{22}=k_{22},$$
and
\begin{align*}
\frac{\partial(K_{1},K_{2})}{\partial(\Phi_{1},\Phi_{3})}\Big|_{\mathbf{\Phi}=0}=\frac{\partial(J_{1},J_{2})}
{\partial(\nu,s,u)}\Big|_{\bar{\mathbf{\Phi}}=0}\cdot\frac{\partial(\nu,s,u)}{\partial(\Phi_{1},\Phi_{3})}\Big|_{\mathbf{\Phi}=0}
=\big(\tilde{\tilde{k}}_{ij}\big)_{2\times2}
\end{align*}
with
$$ \tilde{\tilde{k}}_{11}=k_{11},~~, \tilde{\tilde{k}}_{12}=\tilde{k}_{11},~~\tilde{\tilde{k}}_{21}=k_{21},~~\tilde{\tilde{k}}_{22}=k_{21}.$$
Furthermore, we obtain
\begin{align}
&\det\Big(\frac{\partial(K_{1},K_{2})}{\partial(\Phi_{3},\Phi_{2})}\Big|_{\mathbf{\Phi}=0}\Big)=\tilde{k}_{11}k_{22}-k_{12}k_{21},\label{b44}\\
&\det\Big(\frac{\partial(K_{1},K_{2})}{\partial(\Phi_{1},\Phi_{3})}\Big|_{\mathbf{\Phi}=0}\Big)=-k_{21}.\label{b45}
\end{align}
Finally, it follows from~\eqref{b43}-\eqref{b45}
\begin{align}
&\frac{\partial G_{1}}{\partial\Phi_{3}}\big|_{\mathbf{\Phi}=0}=-\det\Big(\frac{\partial(K_{1},K_{2})}{\partial(\Phi_{3},\Phi_{2})}\Big|_{\mathbf{\Phi}=0}\Big)\cdot
\det\Big(\frac{\partial(K_{1},K_{2})}{\partial(\Phi_{1},\Phi_{2})}\Big|_{\mathbf{\Phi}=0}\Big)^{-1}
=-\frac{\tilde{k}_{11}k_{22}-k_{12}k_{21}}{k_{11}k_{22}-k_{12}k_{21}}
\mathop{=}\limits^{\triangle}h_{1,0},\label{b46}\\
&\frac{\partial G_{2}}{\partial\Phi_{3}}\big|_{\mathbf{\Phi}=0}=-\det\Big(\frac{\partial(K_{1},K_{2})}{\partial(\Phi_{1},\Phi_{3})}\Big|_{\mathbf{\Phi}=0}\Big)\cdot
\det\Big(\frac{\partial(K_{1},K_{2})}{\partial(\Phi_{1},\Phi_{2})}\Big|_{\mathbf{\Phi}=0}\Big)^{-1}=\frac{k_{21}}{k_{11}k_{22}-k_{12}k_{21}}
\mathop{=}\limits^{\triangle}h_{2,0}.\label{b47}
\end{align}
From~\eqref{a10}, \eqref{b46}, and the definition of $k_{ij}(i,j=1,2)$ and $\tilde{k}_{11}$, we have
\begin{align}
|h_{1,0}|<1.\label{b48}
\end{align}
Thus, we proved~\eqref{bb35}.

\end{proof}

By the Rankine-Hugoniot conditions~\eqref{a9}, we can calculate the equations satisfied by the shock curve $x=\chi(t)$ as follows
\begin{align}\label{b49}
\left\{\begin{aligned}
&\frac{d\chi(t)}{dt}=F(\mathbf{\Phi}(t,\chi(t))),\\
&\chi(0)=0,
\end{aligned}\right.
\end{align}
where
\begin{align*}
F(\mathbf{\Phi})=&\sqrt{-\frac{(\gamma+1)e^{\Phi_{2}+s_{\iota}}-\gamma e^{\Phi_{2}+s_{\iota}}\nu(\mathbf{\Phi})-e^{s_{0}}\nu_{0}^{-\gamma}}{\nu(\mathbf{\Phi})-\nu_{0}}}.
\end{align*}
By the equation~$\eqref{a1}_{4}$, \eqref{a2} and the boundary condition~\eqref{b34}, we get
\begin{align}
Z=e^{-\kappa\big(t-\chi^{-1}(x)\big)}<1,\label{BBB49}
\end{align}
where $t=\chi^{-1}(x)$ is the inverse function of $x=\chi(t)$.

Next, we introduce the scaling transformation
\begin{align}
\hat{\mathbf{\Phi}}=(\hat{\Phi}_{1},\hat{\Phi}_{2},\hat{\Phi}_{3})^{\top}=(\frac{1}{\alpha}\Phi_{1},\beta\Phi_{2},\Phi_{3})^{\top},\label{bb58}
\end{align}
with
\begin{align}
|h_{1,0}|<\alpha<1,\quad 0<\beta<\big(|h_{2,0}|\big)^{-1},\label{b58}
\end{align}
then the free boundary value problems~\eqref{b27}-\eqref{b33} become
\begin{align}
\frac{\partial\hat{\Phi}_{1}}{\partial t}+\lambda_{1}(\hat{\Phi}_{2})\frac{\partial\hat{\Phi}_{1}}{\partial x}=&-\frac{1}{\alpha\beta}b_{2}(\hat{\mathbf{\Phi}})\big(\frac{\partial\hat{\Phi}_{2}}{\partial t}+\lambda_{1}(\hat{\Phi}_{2})\frac{\partial\hat{\Phi}_{2}}{\partial x}\big)-\frac{1}{\alpha}b_{3}(\hat{\Phi}_{2})\big(\frac{\partial\hat{\Phi}_{3}}{\partial t}+\lambda_{1}(\hat{\Phi}_{2})\frac{\partial\hat{\Phi}_{3}}{\partial x}\big)\notag\\
&+\frac{1}{\alpha}\Big(\big(1+b_{3}(\hat{\Phi}_{2})\big)\sqrt{\gamma e^{s_{\iota}}}\big(\frac{\gamma+1}{\gamma}-\nu_{\iota}\big)+b_{2}(\hat{\mathbf{\Phi}})\Big)\frac{1}{T}\kappa\hbar e^{-\kappa\big(t-\chi^{-1}(x)\big)},\label{b50}\\
\frac{\partial\hat{\Phi}_{2}}{\partial t}=&\beta\frac{1}{T}\kappa\hbar e^{-\kappa\big(t-\chi^{-1}(x)\big)},\label{b51}\\
\frac{\partial\hat{\Phi}_{3}}{\partial t}+\lambda_{3}(\hat{\Phi}_{2})\frac{\partial\hat{\Phi}_{3}}{\partial x}=&-\alpha b_{3}(\hat{\Phi}_{2})\big(\frac{\partial\hat{\Phi}_{1}}{\partial t}+\lambda_{3}(\hat{\Phi}_{2})\frac{\partial\hat{\Phi}_{1}}{\partial x}\big)-\frac{1}{\beta}b_{2}(\hat{\mathbf{\Phi}})\big(\frac{\partial\hat{\Phi}_{2}}{\partial t}+\lambda_{3}(\hat{\Phi}_{2})\frac{\partial\hat{\Phi}_{2}}{\partial x}\big)\notag\\
&+\Big(\big(1+b_{3}(\hat{\Phi}_{2})\big)\sqrt{\gamma e^{s_{\iota}}}\big(\frac{\gamma+1}{\gamma}-\nu_{\iota}\big)+b_{2}(\hat{\mathbf{\Phi}})\Big)\frac{1}{T}\kappa\hbar e^{-\kappa\big(t-\chi^{-1}(x)\big)}.\label{b52}
\end{align}
\begin{align}
x=0:~~&\hat{\Phi}_{3}(t,0)=\alpha\hat{\Phi}_{1}(t,0)+2(B'(t)-b_{0}),\label{b54}
\end{align}
and
\begin{align}
x=\chi(t):~~&\hat{\Phi}_{1}(t,\chi(t))=\frac{1}{\alpha}G_{1}\big(\hat{\Phi}_{3}(t,\chi(t))\big),\label{b55}\\
&\hat{\Phi}_{2}(t,\chi(t))=\beta G_{2}\big(\hat{\Phi}_{3}(t,\chi(t))\big).\label{b56}
\end{align}

Besides, the equations~\eqref{b49} become
\begin{align}\label{BB49}
\left\{\begin{aligned}
&\frac{d\chi(t)}{dt}=\hat{F}(\hat{\mathbf{\Phi}}(t,\chi(t))),\\
&\chi(0)=0,
\end{aligned}\right.
\end{align}
where $\hat{F}(\hat{\mathbf{\Phi}})=\sqrt{-\frac{(\gamma+1)e^{\frac{1}{\beta}\hat{\Phi}_{2}+s_{\iota}}-\gamma e^{\frac{1}{\beta}\hat{\Phi}_{2}+s_{\iota}}\nu(\hat{\mathbf{\Phi}})-e^{s_{0}}\nu_{0}^{-\gamma}}{\nu(\hat{\mathbf{\Phi}})-\nu_{0}}}$ with
$\nu(\hat{\mathbf{\Phi}})=-\frac{1}{2\sqrt{\gamma e^{s_{\iota}}}}(\alpha\hat{\Phi}_{1}+\hat{\Phi}_{3})+(\frac{\gamma+1}{\gamma}-\nu_{\iota})\frac{1}{\beta}\hat{\Phi}_{2}+\nu_{\iota}$.

\begin{remark}
The estimate~\eqref{bb35} captures an inherent dissipative property of the solution at the shock wave and is essential for establishing uniform a priori estimates in the subsequent proof.
\end{remark}

\begin{remark}
The scaling transformation~\eqref{bb58}, in conjunction with~\eqref{b58} and boundary conditions~\eqref{b54}--\eqref{b56}, reveals that both boundaries of the domain $\hat{D}_{l}=\{(t,x)|t\in\mathbb{R}_{+},0\leq x\leq\chi(t)\}$ possess dissipative structures.
\end{remark}

\section{Existence of the detonation wave solution}\label{s3}
\indent\indent In this section, we will give the proof of~\theref{t1}. In fact, through the preliminaries in~\secref{s2}, we know that to prove~\theref{t1}, we only need to prove the following theorem.
\begin{theorem}\label{t2}
There exists a sufficiently small constant $\hbar^{*}>0$ and a suitably small constant $\epsilon_{1}\in(0,\epsilon_{0})$, such that for any given $\epsilon\in(0,\epsilon_{1})$ and any given $\hbar\in(0,\hbar^{*})$, if the piston velocity $B'(t)$ satisfies $B'(0)=b_{0}$ and~\eqref{a11}, the free boundary value problems~\eqref{b50}-\eqref{b56} admit a unique global $C^{1}$ solution $\hat{\mathbf{\Phi}}=\hat{\mathbf{\Phi}}(t,x)$ satisfying on the angular domain $\hat{D}_{l}$
\begin{align}
&\|\hat{\mathbf{\Phi}}(t,x)\|_{C^{1}(\hat{D}_{l})}\leq M\epsilon. \label{c1}
\end{align}
Moreover, the shock curve $x=\chi(t)$ is a $C^{2}$ function satisfying
\begin{align}
\|\chi'(t)-\chi_{0}\|_{C^{1}(\mathbb{R}_{+})}\leq M_{F}\epsilon,\label{c2}
\end{align}
where $M, M_{F}$ are two positive constants, which is independent of $\epsilon$.
\end{theorem}
\begin{proof}
According to the local existence and uniqueness of classical solution (c.f.~\cite{Yuw}), the free boundary value problems~\eqref{b50}-\eqref{b56} admit a unique classical solution: $\hat{\mathbf{\Phi}}(t,x)\in C^{1}$ and $\chi(t)\in C^{2}$ at least on a local domain $\hat{D}_{l}(\delta)=\{(t,x)|0<t\leq\delta, 0\leq x\leq\chi(t)\}$ with a sufficiently small constant $\delta>0$. Thus, in order to get the global existence and uniqueness of classical solution, it is only necessary to prove that if~\eqref{a11} holds for suitably small $\epsilon>0$, then there exists a positive constant $M$ such that the uniform a priori estimate~\eqref{c1} holds on the whole existence domain $\hat{D}_{l}$ of $C^{1}$ solution.

Let the set
$$\mathcal{F}=\{\sigma\in C^{2}(\mathbb{R}_{+})|\sigma(0)=0, \|\sigma'(t)-\chi_{0}\|_{C^{1}(\mathbb{R}_{+})}\leq M_{F}\epsilon\},$$
then we take any $\chi(t)\in\mathcal{F}$ and then consider the fixed boundary value problem~\eqref{b50}-\eqref{b56}.

For the time being, we suppose
\begin{align}
&\|\hat{\mathbf{\Phi}}(t,x)\|_{C^{1}(\hat{D}_{l})}\leq M_{0}\epsilon_{1}. \label{c3}
\end{align}
This reasonableness of this hypothesis will be explained at the end of the proof.

Obviously,
\begin{align}
\lambda_{1}(\hat{\Phi}_{2})<0<\lambda_{3}(\hat{\Phi}_{2}).\label{c4}
\end{align}
Let $\eta_{1}=\eta_{1}(\tau;t,x)$ and $\eta_{3}=\eta_{3}(\tau;t,x)$ be the $1$st and $3$rd characteristic curve passing through a point $(t,x)$ respectively, which satisfy the following equations
\begin{align}\label{c6}
\left\{\begin{aligned}
&\frac{d\eta_{1}(\tau;t,x)}{d\tau}=\lambda_{1}(\hat{\Phi}_{2}(\tau,\eta_{1}(\tau;t,x))),\\
&\tau=t:~~\eta_{1}(t;t,x)=x,
\end{aligned}\right.
\end{align}
and
\begin{align}\label{cc6}
\left\{\begin{aligned}
&\frac{d\eta_{3}(\tau;t,x)}{d\tau}=\lambda_{3}(\hat{\Phi}_{2}(\tau,\eta_{3}(\tau;t,x))),\\
&\tau=t:~~\eta_{3}(t;t,x)=x.
\end{aligned}\right.
\end{align}
Now, we provide the proof of estimate~\eqref{c1}. Specifically, we will prove that the following estimates hold:
\begin{align}
\|\hat{\mathbf{\Phi}}(t,x)\|_{C^{0}(\hat{D}_{l})}\leq& M_{1}\epsilon, \label{M1}\\
\mathop{\max}\limits_{i=1,2,3}\|\partial_{t}\hat{\Phi}_{i}(t,x)\|_{C^{0}(\hat{D}_{l})}\leq& M_{1}\epsilon, \label{M2}\\
\max\big\{\|\partial_{x}\hat{\Phi}_{1}(t,x)\|_{C^{0}(\hat{D}_{l})},\|\partial_{x}\hat{\Phi}_{3}(t,x)&\|_{C^{0}(\hat{D}_{l})}\big\}\leq M_{2}\epsilon, \label{M3}\\
\|\partial_{x}\hat{\Phi}_{2}(t,x)\|_{C^{0}(\hat{D}_{l})}\leq& M_{3}\epsilon, \label{M4}
\end{align}
where $M_{i}(i=1,2,3)$ are some positive constants, which is less than $M$ and will be determined later.

Obviously, the estimates~\eqref{M1}-\eqref{M4} hold on $\hat{D}_{l}(\delta_{0})=\{(t,x)|0<t\leq\delta_{0}, 0\leq x\leq\chi(t)\}$ when $\delta_{0}>0$ is suitably small. Now suppose that~\eqref{M1}-\eqref{M4} hold on a domain $\hat{D}_{l}(T_{0})=\{(t,x)|0<t\leq T_{0}, 0\leq x\leq\chi(t)\}$, we will prove that for given $\delta_{1}>0$ independent of $T_{0}$, \eqref{M1}-\eqref{M4} hold on $\hat{D}_{l}(T_{0}+\delta_{1})$, provided that the $C^{1}$ solution exists on $\hat{D}_{l}(T_{0}+\delta_{1})$.

Let
$$\partial_{1}=\partial_{t}+\lambda_{1}(\hat{\Phi}_{2})\partial_{x},\quad \partial_{3}=\partial_{t}+\lambda_{3}(\hat{\Phi}_{2})\partial_{x},$$
and
$$\partial_{\hat{F}}=\partial_{t}+\hat{F}(\hat{\mathbf{\Phi}})\partial_{x}.$$
Then, we have
\begin{align}
&\partial_{\hat{F}}=\frac{\lambda_{1}(\hat{\Phi}_{2})-\hat{F}(\hat{\mathbf{\Phi}})}{\lambda_{1}(\hat{\Phi}_{2})}\partial_{t}
+\frac{\hat{F}(\hat{\mathbf{\Phi}})}{\lambda_{1}(\hat{\Phi}_{2})}\partial_{1}
=\frac{\lambda_{3}(\hat{\Phi}_{2})-\hat{F}(\hat{\mathbf{\Phi}})}{\lambda_{3}(\hat{\Phi}_{2})}\partial_{t}
+\frac{\hat{F}(\hat{\mathbf{\Phi}})}{\lambda_{3}(\hat{\Phi}_{2})}\partial_{3},\label{v1}\\
&\partial_{x}=\frac{1}{\hat{F}(\hat{\mathbf{\Phi}})}(\partial_{\hat{F}}-\partial_{t}).\label{v2}
\end{align}
Moreover, equations~\eqref{b50}-\eqref{b52} become
\begin{align}
\partial_{1}\hat{\Phi}_{1}=&-\frac{1}{\alpha\beta}b_{2}(\hat{\mathbf{\Phi}})\partial_{1}\hat{\Phi}_{2}
-\frac{1}{\alpha}b_{3}(\hat{\Phi}_{2})\partial_{1}\hat{\Phi}_{3}+\frac{1}{\alpha}\Big(\big(1+b_{3}(\hat{\Phi}_{2})\big)\sqrt{\gamma e^{s_{\iota}}}\big(\frac{\gamma+1}{\gamma}-\nu_{\iota}\big)\notag\\
&+b_{2}(\hat{\mathbf{\Phi}})\Big)\frac{1}{T}\kappa\hbar e^{-\kappa\big(t-\chi^{-1}(x)\big)},\label{V1}\\
\partial_{t}\hat{\Phi}_{2}=&\beta \frac{1}{T}\kappa\hbar e^{-\kappa\big(t-\chi^{-1}(x)\big)},\label{V2}\\
\partial_{3}\hat{\Phi}_{3}=&-\alpha b_{3}(\hat{\Phi}_{2})\partial_{3}\hat{\Phi}_{1}-\frac{1}{\beta} b_{2}(\hat{\mathbf{\Phi}})\partial_{3}\hat{\Phi}_{2}+\Big(\big(1+b_{3}(\hat{\Phi}_{2})\big)\sqrt{\gamma e^{s_{\iota}}}\big(\frac{\gamma+1}{\gamma}-\nu_{\iota}\big)\notag\\
&+b_{2}(\hat{\mathbf{\Phi}})\Big)\frac{1}{T}\kappa\hbar e^{-\kappa\big(t-\chi^{-1}(x)\big)}.\label{V3}
\end{align}

For any fixed $(t,x)\in\hat{D}_{l}(T_{0}+\delta_{1})$, there exists a unique $\tau_{i}=\tau_{i}(t,x)<t (i=1,2,3)$, such that
\begin{align}
\eta_{1}(\tau_{1};t,x)=\chi(\tau_{1}),~~x=\chi(\tau_{2}),~~\eta_{3}(\tau_{3};t,x)=0.\label{c7}
\end{align}
Let
$$\tau_{3,1}=\tau_{3,1}(t,x)=\tau_{1}(\tau_{3}(t,x),0),$$
and $0\leq\tau_{3,1}\leq T_{0}$.

Firstly, we integrate~\eqref{V3} along the $3$rd characteristic curve $\eta_{3}$ from $(\tau_{3},0)$ to $(t,x)$ and then integrate~\eqref{V1} along the $1$st characteristic curve $\eta_{1}$ from $(\tau_{3,1},\chi(\tau_{3,1}))$ to $(\tau_{3},0)$ to obtain
\begin{align}
|\hat{\Phi}_{3}(t,x)|\leq&\big|\hat{\Phi}_{3}\big(\tau_{3},0\big)\big|+\Big|\int_{\tau_{3}}^{t}\Big(-\alpha b_{3}(\hat{\Phi}_{2})\partial_{3}\hat{\Phi}_{1}-\frac{1}{\beta} b_{2}(\hat{\mathbf{\Phi}})\partial_{3}\hat{\Phi}_{2}+\Big(\big(1+b_{3}(\hat{\Phi}_{2})\big)\sqrt{\gamma e^{s_{\iota}}}\big(\frac{\gamma+1}{\gamma}-\nu_{\iota}\big)\notag\\
&+b_{2}(\hat{\mathbf{\Phi}})\Big)\frac{1}{T}\kappa\hbar e^{-\kappa\big(\tau-\chi^{-1}(\eta_{3}(\tau;t,x))\big)}\Big)(\tau,\eta_{3}(\tau;t,x))d\tau\Big|\notag\\
\leq&\alpha\big|\hat{\Phi}_{1}\big(\tau_{3},0\big)\big|+2|B'(\tau_{3})-b_{0}|+C\kappa\hbar\int_{\tau_{3}}^{t}
e^{-\kappa\big(\tau-\chi^{-1}(\eta_{3}(\tau;t,x))\big)}d\tau+C\epsilon_{1}^{2}\notag\\
\leq&\alpha\big|\hat{\Phi}_{1}\big(\tau_{3,1},\chi(\tau_{3,1})\big)\big|+2|B'(\tau_{3})-b_{0}|
+C\kappa\hbar\int_{\tau_{3,1}}^{\tau_{3}}e^{-\kappa\big(\tau-\chi^{-1}(\eta_{1}(\tau;t,x))\big)}d\tau\notag\\
&+C\kappa\hbar\int_{\tau_{3}}^{t}e^{-\kappa\big(\tau-\chi^{-1}(\eta_{3}(\tau;t,x))\big)}d\tau+C\epsilon_{1}^{2}\notag\\
\leq&\big|G_{1}\big(\hat{\Phi}_{3}\big(\tau_{3,1},\chi(\tau_{3,1})\big)\big)\big|+2|B'(\tau_{3})-b_{0}|
+C\kappa\hbar\int_{\tau_{3,1}}^{\tau_{3}}e^{-\kappa(\tau-\tau_{3,1})}d\tau\notag\\
&+C\kappa\hbar\int_{\tau_{3}}^{t}e^{-\kappa(\tau-\tau_{3})}d\tau+C\epsilon_{1}^{2}\notag\\
\leq&|h_{1,0}|M_{1}\epsilon+2\epsilon+C\hbar+C\epsilon_{1}^{2}\notag\\
\leq&M_{1}\epsilon,\label{c9}
\end{align}
where $C>0$ is a generic constant, which is different in different places.~Here we used~\eqref{a11}, \eqref{b26}, \eqref{b48}, \eqref{c3} and $\hbar$ sufficiently small.

Then, integrating~\eqref{V1} along the $1$st characteristic curve $\eta_{1}$ from $\big(\tau_{1},\chi(\tau_{1})\big)$ to $(t,x)$, and using~\eqref{b26},\eqref{b58},\eqref{c3} and~\eqref{c9}, we get
\begin{align}
|\hat{\Phi}_{1}(t,x)|\leq&\frac{1}{\alpha}\big|G_{1}\big(\hat{\Phi}_{3}\big(\tau_{1},\chi(\tau_{1})\big)\big)\big|
+C\kappa\hbar\int_{\tau_{1}}^{t}e^{-\kappa(\tau-\tau_{1})}d\tau+C\epsilon_{1}^{2}\notag\\
\leq&\frac{|h_{1,0}|}{\alpha}M_{1}\epsilon+C\hbar+C\epsilon_{1}^{2}\notag\\
\leq&M_{1}\epsilon.\label{c10}
\end{align}
Integrating~\eqref{V2} from $\big(\tau_{2},\chi(\tau_{2})\big)$ to $(t,x)$ and by~\eqref{b58},\eqref{c3} and~\eqref{c9}, we have
\begin{align}
|\hat{\Phi}_{2}(t,x)|\leq&\beta|h_{2,0}|\big|\hat{\Phi}_{3}\big(\tau_{2},\chi(\tau_{2})\big)\big|+\beta\frac{1}{T}\kappa\hbar \int_{\tau_{2}}^{t} e^{-\kappa\big(\tau-\tau_{2}\big)}d\tau\notag\\
\leq&\beta|h_{2,0}|\big|\hat{\Phi}_{3}\big(\tau_{2},\chi(\tau_{2})\big)\big|+C\hbar\notag\\
\leq&\beta|h_{2,0}|M_{1}\epsilon+C\hbar\notag\\
\leq&M_{1}\epsilon.\label{c11}
\end{align}
Thus, we proved~\eqref{M1}.

Taking the derivative of equations~\eqref{V1} and~\eqref{V3} with respect to $t$ and using the equation~\eqref{V2}, we have
\begin{align}
\partial_{1}\big(\partial_{t}\hat{\Phi}_{1}\big)
=&-\frac{1}{\alpha\beta}\sum_{i=1}^{3}\frac{\partial b_{2}}{\partial\hat{\Phi}_{i}}\partial_{t}\hat{\Phi}_{i}\partial_{1}\hat{\Phi}_{2}
-\frac{1}{\alpha\beta}b_{2}(\hat{\mathbf{\Phi}})\partial_{1}\big(\partial_{t}\hat{\Phi}_{2}\big)
-\frac{1}{\alpha}\frac{\partial b_{3}}{\partial\hat{\Phi}_{2}}\partial_{t}\hat{\Phi}_{2}\partial_{1}\hat{\Phi}_{3}\notag\\
&-\frac{1}{\alpha}b_{3}(\hat{\Phi}_{2})\partial_{1}\big(\partial_{t}\hat{\Phi}_{3}\big)+R_{1}\kappa\hbar e^{-\kappa\big(t-\chi^{-1}(x)\big)},
\label{c16}\\
\partial_{3}\big(\partial_{t}\hat{\Phi}_{3}\big)
=&-\alpha\frac{\partial b_{3}}{\partial\hat{\Phi}_{2}}\partial_{t}\hat{\Phi}_{2}\partial_{3}\hat{\Phi}_{1}-\alpha b_{3}(\hat{\Phi}_{2})\partial_{3}\big(\partial_{t}\hat{\Phi}_{1}\big)-\frac{1}{\beta}\sum_{i=1}^{3}\frac{\partial b_{2}}{\partial\hat{\Phi}_{i}}\partial_{t}\hat{\Phi}_{i}\partial_{3}\hat{\Phi}_{2}\notag\\
&-\frac{1}{\beta} b_{2}(\hat{\mathbf{\Phi}})\partial_{3}\big(\partial_{t}\hat{\Phi}_{2}\big)+R_{3}\kappa\hbar e^{-\kappa\big(t-\chi^{-1}(x)\big)},
\label{c17}
\end{align}
where
\begin{align*}
R_{1}=&-\Big(\frac{\partial\lambda_{1}}{\partial\hat{\Phi}_{2}}\partial_{x}\hat{\Phi}_{1}+\frac{1}{\alpha\beta}b_{2}(\hat{\mathbf{\Phi}})
\frac{\partial\lambda_{1}}{\partial\hat{\Phi}_{2}}\partial_{x}\hat{\Phi}_{2}+\frac{1}{\alpha}b_{3}(\hat{\Phi}_{2})
\frac{\partial\lambda_{1}}{\partial\hat{\Phi}_{2}}\partial_{x}\hat{\Phi}_{3}\Big)\beta\frac{1}{T}\notag\\
&+\frac{1}{\alpha}\Big(\frac{\partial b_{3}}{\partial\hat{\Phi}_{2}}\partial_{t}\hat{\Phi}_{2}\sqrt{\gamma e^{s_{\iota}}}\big(\frac{\gamma+1}{\gamma}-\nu_{\iota}\big)+\sum_{i=1}^{3}\frac{\partial b_{2}}{\partial\hat{\Phi}_{i}}\partial_{t}\hat{\Phi}_{i}\Big)\frac{1}{T}\\
&-\frac{1}{\alpha}\Big(\big(1+b_{3}(\hat{\Phi}_{2})\big)
\sqrt{\gamma e^{s_{\iota}}}\big(\frac{\gamma+1}{\gamma}-\nu_{\iota}\big)+b_{2}(\hat{\mathbf{\Phi}})\Big)\frac{1}{T}\kappa\notag\\
&-\frac{1}{\alpha}\Big(\big(1+b_{3}(\hat{\Phi}_{2})\big)
\sqrt{\gamma e^{s_{\iota}}}\big(\frac{\gamma+1}{\gamma}-\nu_{\iota}\big)+b_{2}(\hat{\mathbf{\Phi}})\Big)\frac{1}{T^{2}}\partial_{t}T,\\
R_{3}=&-\Big(\frac{\partial\lambda_{3}}{\partial\hat{\Phi}_{2}}\partial_{x}\hat{\Phi}_{3}+\alpha b_{3}(\hat{\Phi}_{2})
\frac{\partial\lambda_{3}}{\partial\hat{\Phi}_{2}}\partial_{x}\hat{\Phi}_{1}+\frac{1}{\beta}b_{2}(\hat{\mathbf{\Phi}})
\frac{\partial\lambda_{3}}{\partial\hat{\Phi}_{2}}\partial_{x}\hat{\Phi}_{2}\Big)\beta\frac{1}{T}\notag\\
&+\Big(\frac{\partial b_{3}}{\partial\hat{\Phi}_{2}}\partial_{t}\hat{\Phi}_{2}\sqrt{\gamma e^{s_{\iota}}}\big(\frac{\gamma+1}{\gamma}-\nu_{\iota}\big)+\sum_{i=1}^{3}\frac{\partial b_{2}}{\partial\hat{\Phi}_{i}}\partial_{t}\hat{\Phi}_{i}\Big)\frac{1}{T}\\
&-\Big(\big(1+b_{3}(\hat{\Phi}_{2})\big)
\sqrt{\gamma e^{s_{\iota}}}\big(\frac{\gamma+1}{\gamma}-\nu_{\iota}\big)+b_{2}(\hat{\mathbf{\Phi}})\Big)\frac{1}{T}\kappa\notag\\
&-\Big(\big(1+b_{3}(\hat{\Phi}_{2})\big)
\sqrt{\gamma e^{s_{\iota}}}\big(\frac{\gamma+1}{\gamma}-\nu_{\iota}\big)+b_{2}(\hat{\mathbf{\Phi}})\Big)\frac{1}{T^{2}}\partial_{t}T.
\end{align*}
Moreover, it follows from~\eqref{b26} and~\eqref{c3}
\begin{align}
R_{1}<0,\quad R_{3}<0.\label{RR1}
\end{align}
Similar to~\eqref{c9}, integrating~\eqref{c17} along the $3$rd characteristic curve $\eta_{3}$ from $(\tau_{3},0)$ to $(t,x)$ and then integrating~\eqref{c16} along the $1$st characteristic curve $\eta_{1}$ from $(\tau_{3,1},\chi(\tau_{3,1}))$ to $(\tau_{3},0)$, we have
\begin{align}
\big|\partial_{t}\hat{\Phi}_{3}(t,x)\big|
\leq&|h_{1,0}|M_{1}\epsilon+2\epsilon+C\epsilon_{1}\epsilon\notag\\
\leq&M_{1}\epsilon.\label{c18}
\end{align}
Here we used~\eqref{a11},\eqref{b26},\eqref{b48},\eqref{c3},\eqref{M1} and~\eqref{RR1}.

By~\eqref{v1} and~\eqref{b26},\eqref{b55},\eqref{c3},\eqref{M1},\eqref{V1},\eqref{V3},\eqref{c18}, we obtain at the boundary $x=\chi(t)$
\begin{align}
|\partial_{t}\hat{\Phi}_{1}(t,\chi(t))|\leq&\Big|\frac{\lambda_{1}(\hat{\Phi}_{2})}{\lambda_{1}(\hat{\Phi}_{2})-\hat{F}(\hat{\mathbf{\Phi}})}
\Big(\partial_{\hat{F}}\hat{\Phi}_{1}(t,\chi(t))
-\frac{\hat{F}(\hat{\mathbf{\Phi}})}{\lambda_{1}(\hat{\Phi}_{2})}\partial_{1}\hat{\Phi}_{1}(t,\chi(t))\Big)\Big|\notag\\
\leq&\frac{\lambda_{b}+C\epsilon}{\lambda_{b}+\chi_{0}-C\epsilon}
\frac{|\hbar_{1,0}|}{\alpha}\big|\partial_{\hat{F}}\hat{\Phi}_{3}(t,\chi(t))\big|+C\epsilon_{1}\epsilon+C\hbar\notag\\
\leq&\frac{\lambda_{b}+C\epsilon}{\lambda_{b}+\chi_{0}-C\epsilon}\frac{|\hbar_{1,0}|}{\alpha}
\frac{\lambda_{b}-\chi_{0}+C\epsilon}{\lambda_{b}-C\epsilon}\|\partial_{t}\hat{\Phi}_{3}\|_{C^{0}}+C\epsilon_{1}\epsilon+C\hbar
\notag\\
\leq&\frac{|\hbar_{1,0}|}{\alpha}M_{1}\epsilon,\label{c19}
\end{align}
where the constant $\lambda_{b}=\sqrt{\gamma e^{s_{\iota}}}$.
Then, integrating~\eqref{c16} along the $1$st characteristic curve $\eta_{1}$ from $\big(\tau_{1},\chi(\tau_{1})\big)$ to $(t,x)$, and using~\eqref{b26},\eqref{b58},\eqref{c3},\eqref{M1} and~\eqref{c19}, we get
\begin{align}
|\partial_{t}\hat{\Phi}_{1}(t,x)|\leq&|\partial_{t}\hat{\Phi}_{1}(\tau_{1},\chi(\tau_{1}))|+C\epsilon_{1}\epsilon\notag\\
\leq&\frac{|\hbar_{1,0}|}{\alpha}M_{1}\epsilon+C\epsilon_{1}\epsilon\notag\\
\leq&M_{1}\epsilon.\label{c20}
\end{align}
By the equation~\eqref{v2}, we obtain
\begin{align}
|\partial_{t}\hat{\Phi}_{2}(t,x)|\leq\Big|\beta \frac{1}{T}\kappa\hbar e^{-\kappa\big(t-\tau_{2}\big)}\Big|\leq C\hbar\leq M_{1}\epsilon.\label{c23}
\end{align}
Thus, we proved~\eqref{M2}.

By equations~\eqref{V1} and~\eqref{V3}, and~\eqref{b26},\eqref{c3},\eqref{M1},\eqref{c18},\eqref{c20}, we have
\begin{align}
|\partial_{x}\hat{\Phi}_{1}(t,x)|
\leq&\mu_{\max}M_{1}\epsilon+C\epsilon_{1}\epsilon+C\hbar
\leq M_{2}\epsilon,\label{c21}\\
|\partial_{x}\hat{\Phi}_{3}(t,x)|\leq&M_{2}\epsilon,\label{c22}
\end{align}
where constants $\mu_{\max}=\mathop{\sup}\limits_{(t,x)}\frac{1}{\lambda_{3}(\hat{\Phi}_{2}(t,x))}$ and $M_{2}>\mu_{max}M_{1}$. Thus, we proved~\eqref{M3}.

Taking the derivative of equations~\eqref{V2} with respect to $x$, one has
\begin{align}
\partial_{t}\Big(\partial_{x}\hat{\Phi}_{2}\Big)=-\beta\frac{1}{T^{2}}\partial_{x}T\kappa\hbar e^{-\kappa\big(t-\chi^{-1}(x)\big)}
+\beta\frac{1}{T}\kappa^{2}\frac{1}{\chi'(t)}\hbar e^{-\kappa\big(t-\chi^{-1}(x)\big)}.\label{c24}
\end{align}
Then, by~\eqref{b26},\eqref{b58},\eqref{b56},\eqref{c3},\eqref{M1},\eqref{v2},\eqref{V2}-\eqref{V3} and~\eqref{c18}, we obtain
\begin{align}
\Big|\partial_{x}\hat{\Phi}_{2}(t,x)\Big|\leq&\Big|\partial_{x}\hat{\Phi}_{2}(\tau_{2},\chi(\tau_{2}))\Big|+C\kappa\hbar
\int_{\tau_{2}}^{t}e^{-\kappa\big(\tau-\tau_{2}\big)}d\tau\notag\\
\leq&\frac{1}{\chi_{0}-C\epsilon}\Big(\big|\partial_{\hat{F}}\hat{\Phi}_{2}(\tau_{2},\chi(\tau_{2}))\big|
+\big|\partial_{t}\hat{\Phi}_{2}(\tau_{2},\chi(\tau_{2}))\big|\Big)+C\hbar\notag\\
\leq&\frac{1}{\chi_{0}-C\epsilon}\beta|h_{2,0}|\big|\partial_{\hat{F}}\hat{\Phi}_{3}(\tau_{2},\chi(\tau_{2}))\big|+C\hbar\notag\\
\leq&\frac{1}{\chi_{0}-C\epsilon}\beta|h_{2,0}|\frac{\lambda_{b}-\chi_{0}+C\epsilon}{\lambda_{b}-C\epsilon}
\|\partial_{t}\hat{\Phi}_{3}\|_{C^{0}}+C\epsilon_{1}\epsilon+C\hbar\notag\\
\leq&\frac{1}{\chi_{0}-C\epsilon}\beta|h_{2,0}|\frac{\lambda_{b}-\chi_{0}+C\epsilon}{\lambda_{b}-C\epsilon}
M_{1}\epsilon+C\epsilon_{1}\epsilon+C\hbar\notag\\
\leq&M_{3}\epsilon,\label{c25}
\end{align}
where the constant $M_{3}>\frac{1}{\chi_{0}}M_{1}$. Thus, we proved~\eqref{M4}. Furthermore, we obtain that there exists a positive constant $M$ such that~\eqref{c1} holds.~And the previous hypothesis~\eqref{c3} is actually reasonable for $M_{0}>\mathop{\max}\limits_{i=1,2,3}\{M_{i}\}$, provided that $\epsilon$ is suitably small.

By~\eqref{BB12}, the equations~\eqref{BB49} and estimate~\eqref{c1}, we have
\begin{align}
&\|\chi'(t)-\chi_{0}\|_{C^{0}(\mathbb{R}_{+})}\leq\big|\hat{F}\big(\hat{\mathbf{\Phi}}(t,\chi(t))\big)-\chi_{0}\big|\leq M_{F,1}\epsilon,\label{c27}\\
&\|\chi''(t)\|_{C^{0}(\mathbb{R}_{+})}\leq\Big|\sum_{i=1}^{3}\frac{\partial\hat{F}}{\partial\hat{\Phi}_{i}}
\partial_{t}\hat{\Phi}_{i}(t,\chi(t))\Big|\leq M_{F,2}\epsilon.\label{c28}
\end{align}
Then, taking the constant $M_{F}>\max\{M_{F,1},M_{F,2}\}$, one has $\chi(t)\in \mathcal{F}$ and estimate~\eqref{c2} holds. Thus, we finish the proof of~\theref{t2}.

\end{proof}

\section*{Acknowledgement}
 \indent\indent Huimin Yu is supported in part by NSFC Grant No. 12271310 and Natural Science Foundation of Shandong Province ZR2022MA088.

\end{sloppypar}
\end{document}